\documentclass[reqno,tbtags,sumlimits,intlimits,a4paper,oneside,11pt]{amsart}
\def\x{2.95}
\usepackage[lmargin=\x cm,rmargin=\x cm,tmargin=\x cm,bmargin=\x cm,bindingoffset=0cm]{geometry}
\usepackage{amsthm}
\usepackage{amsmath}
\usepackage{amssymb}
\usepackage{amsfonts}
\usepackage{mathrsfs}
\usepackage{tikz}
\usetikzlibrary{decorations.markings}
\usetikzlibrary{cd}
\usepackage{url}
\usepackage{color} 
\usepackage{blkarray}
\usepackage[shortlabels]{enumitem}

\usepackage{textcomp}
\usepackage{mathscinet}
\usepackage{etoolbox}
\usepackage{mathtools}
\usepackage[backend=biber, 
    style=numeric,
    sorting=nyt,
    date=year, 
    hyperref=false,
    arxiv=abs,  
    url=false, 
    doi,
    eprint, 
    giveninits,
    ]{biblatex}
\appto{\bibsetup}{\sloppy}

\newcommand{\mc}[1]{\mathcal{#1}}
\newcommand{\mb}[1]{\mathbb{#1}}
\newcommand{\mf}[1]{\mathfrak{#1}}
\newcommand{\mr}[1]{\mathrm{#1}}

\newcommand{\rd}{\mathrm{d}}
\newcommand{\uh}{\mathrm{H}^2}

\newcommand{\hypsp}[1]{\mathrm{H}^{#1}}
\newcommand{\eucsp}[1]{\mathrm{E}^{#1}}
\newcommand{\sphsp}[1]{\mathrm{S}^{#1}}

\def\beq#1#2\eeq{%
        \begin{equation}%
        \label{#1}%
            #2%
        \end{equation}%
    }
\newcommand{\SLNZ}[1][2]{\mathrm{SL}({#1},\mb{Z})}
\newcommand{\SLNC}[1][n]{\mathrm{SL}({#1},\mb{C})}
\newcommand{\slnc}[1][n]{\mf{sl}({#1},\mb{C})}

\newcommand{\rg}{G}%Reduction group
\newcommand{\ta}{\mf a}%{\End(V)}%target algebra

\newcommand{\Res}[3]{\mathrm{Res}^{#1}_{#2}\left(#3\right)}

\newcommand{\Id}{\mathrm{Id}}

\newcommand{\GL}{\mathrm{GL}}
\newcommand{\SL}{\mathrm{SL}}
\newcommand{\SO}{\mathrm{SO}}
\newcommand{\SU}{\mathrm{SU}}
\newcommand{\NSO}{\mathrm{O}}
\newcommand{\NSU}{\mathrm{U}}

\newcommand{\SE}{\mathrm{SE}}

\newcommand{\PSL}{\mathrm{PSL}}

\newcommand{\PGL}{\mathrm{PGL}}
\newcommand{\PSU}{\mathrm{PSU}}

\newcommand{\Hom}{\mathrm{Hom}}
\newcommand{\End}{\mathrm{End}}
\newcommand{\Aut}[1]{\mathrm{Aut}\!\left(#1 \right)}

\newcommand{\Ad}{\mathrm{Ad}}
\newcommand{\ad}{\mathrm{ad}}

\newcommand{\diag}{\mathrm{diag}}

\newtheorem{Theorem}{Theorem}[section]
\newtheorem{Proposition}[Theorem]{Proposition}
\newtheorem{Lemma}[Theorem]{Lemma}
\newtheorem{Corollary}[Theorem]{Corollary}

\newtheorem{Example}[Theorem]{Example}

\newtheorem{Remark}[Theorem]{Remark}

\newcommand{\field}{\mb{F}}
\newcommand{\mapone}[1]{M(#1)}
\newcommand{\maptwo}[2]{M(#1,#2)}
\newcommand{\mapthree}[3]{M(#1,#2,#3)}
\newcommand{\emap}[3]{M_{#1}(#2,#3)}
\newcommand{\ma}[3]{\mathfrak{M}(#1,#2,#3)}
\newcommand{\ema}[3]{\mathfrak{A}(#1,#2,#3)}
\newcommand{\aaa}[3]{\mathfrak{M}_{#1}(#2,#3)}

\newcommand{\hs}{X}

\title[Computing equivariant matrices on homogeneous spaces]{Computing equivariant matrices on homogeneous spaces for Geometric Deep Learning and Automorphic Lie Algebras}

\author{Vincent Knibbeler}
\address{Department of Mathematical Sciences, Loughborough University, Loughborough LE11 3TU, UK}
\address{Maxwell Institute for Mathematical Sciences, The Bayes Centre, Edinburgh EH8 9BT, UK}
\address{Department of Mathematics, Heriot-Watt University, Edinburgh EH14 4AS, UK}
\email{V.Knibbeler@gmail.com}

\begin{document}

\date{\today}

\begin{abstract}
We develop an elementary method to compute spaces of equivariant maps from a homogeneous space $G/H$ of a Lie group $G$ to a module of this group. The Lie group is not required to be compact. More generally, we study spaces of invariant sections in homogeneous vector bundles, and take a special interest in the case where the fibres are algebras. These latter cases have a natural global algebra structure. We classify these automorphic algebras for the case where the homogeneous space has compact stabilisers.
This work has applications in the theoretical development of geometric deep learning and also in the theory of automorphic Lie algebras. 
\end{abstract}

\maketitle

\vspace{3mm}

\vspace{3mm}

\section{Introduction}
\label{sec:introduction}
Let $G$ be a Lie group and $X$ a homogeneous space of $G$. That is, a smooth manifold with transitive $G$-action: for any two points $x_0,x\in \hs$ there exists $g\in G$ with $g x_0=x$. Let $V$ be a representation of $G$ and denote the space of $G$-equivariant maps $\hs\to V$ by 
$$\emap{G}{\hs}{V}=\{K:\hs\to V\,\vert \,K(gx)=gK(x)\;\forall g\in G,\,\forall x\in\hs\}.$$
We develop an elementary method to compute this space. More generally, for any closed subgroup $H$ of $G$ and representation $W$ of $H$ we compute the $G$-invariant sections of the homogeneous vector bundle $G\times_H W$ over $G/H$, as we explain in detail in Section \ref{sec:vector}. 
In Section \ref{sec:algebra} we replace $V$ by an algebra $\ta$ with compatible $G$-module structure, and study the algebra
$\aaa{\rg}{\hs}{\ta}$
of $\rg$-equivariant maps $\hs\to\ta$.

This investigation is motivated on one hand by recent developments of geometric deep learning, and on the other hand by the expanding horizons of research on automorphic Lie algebras.
The connection between these two topics has not been studied so far, but they are dealing with closely related problems.

This paper can be read without knowledge of deep learning and automorphic Lie algebras, because we only tackle the differential geometric problem described above. Nonetheless, we give a brief description of the two areas in this introduction and explain why they motivate our objective. One can skip to Section \ref{sec:vector} from here. 
\vspace{3mm}

Artificial intelligence is currently developing at a rapid pace and there is a great need for mathematics in this field, as is explained in the recent survey article by Kutyniok \cite{kutyniok2022mathematics}. A particular challenge where researchers are looking to mathematics for answers is the search for good architectures of deep neural networks. One approach to this problem is the highly active field of geometric deep learning.

Two surveys articles on geometric deep learning appeared in May 2021, by Bronstein, Bruna, Cohen and Veli{\v{c}}kovi{\'c} \cite{bronstein2021geometric} and Gerken et al.~\cite{gerken2023geometric} and another by Weiler, Forr{\'e}, Verlinde and Welling \cite{weiler2021coordinate} appeared only a month later. We will not attempt to provide an overview of the literature but simply refer to the combined literature reviews in these three works, which we would not be able to improve. We do mention however the paper of Cohen, Weiler, Kicanaoglu and Welling \cite{cohen2019gauge} and Cheng et al. \cite{cheng2019covariance}, which highlight the importance of convolution with equivariant kernels. These papers inspired this research: we compute spaces of equivariant kernels.

We start our slightly longer summary of the geometric deep learning that motivated this paper with
Cohen, Geiger and Weiler's description of equivariant convolutional networks on homogeneous spaces \cite{cohen2019a}. The features of such convolutional networks are sections of homogeneous vector bundles $G\times_{H} V$, which can be presented as maps $f:G\to V$ such that $f(gh)=h^{-1}f(g)$ for all $h\in H$ (see Section \ref{sec:vector} for more details). 
The authors use the fact that cross-correlation
on the group
\beq{eq:cross}(\kappa\star f)(g)=\int_G\kappa(g^{-1}h)f(h)\rd h
\eeq is equivariant under the group $G$ in the sense that
$$\kappa\star (u\cdot f)=u\cdot (\kappa\star f),\quad \forall u\in G,$$ 
granted that $\rd h$ is a left-invariant measure on $G$\footnote{cross-correlation is sometimes called convolution in the literature, although the term convolution is classically reserved for the operation $(\kappa\ast f)(g)=\int_G\kappa(gh^{-1})f(h)\rd h$, which is not $G$-equivariant.}. The map $\kappa:G\to \Hom(V ,V_{\mr{out}})$ is called the kernel, or the filter, and the group acts as $u\cdot f=f\circ u^{-1}$.
The main result of \cite{cohen2019a} shows that $\kappa\star$ takes sections of $G\times_{H} V_{\mr{in}}$ to sections of $G\times_{H} V_{\mr{out}}$ if and only if $\kappa$ descends to a section of the homogeneous vector bundle $G\times_{H\times H} \Hom(V_{\mr{in}},V_{\mr{out}})$ defined by the right action of $H\times H$ on $G\times  \Hom(V_{\mr{in}},V_{\mr{out}})$ given by $(g,m)\cdot(h_1,h_2)=(h_2^{-1}\,g\,h_1,\,\rho_{\mr{out}}(h_2)^{-1}\,m\,\rho_{\mr{in}}(h_1))$.

Geometric deep learning developed further in 
the papers by Cohen, Weiler, Kicanaoglu and Welling \cite{cohen2019gauge} and Cheng et al. \cite{cheng2019covariance} where one considers convolutional networks on manifolds $M$ in general, rather than homogeneous spaces specifically. This adds another level of abstraction that we briefly describe. 
For a comprehensive treatment, we refer to \cite{weiler2021coordinate}.

Lacking a transitive group action on a manifold $M$, one has to find another way to move a filter over the manifold. The solution proposed in \cite{cohen2019gauge,cheng2019covariance} is to use parallel transport.

The $i$-th feature space is the space of sections $\Gamma_i$ of a vector bundle $\pi_i:E_i\to M$.
Features $f_i$ are elements of $\Gamma_i$.
The kernel is defined by a map
$$K:T_{x^*}M\to\Hom(\pi_{\mr{in}}^{-1}(x^*), \pi_{\mr{out}}^{-1}(x^*))$$
on the tangent space of $M$ at the point $x^*$, and the convolution is then defined by  an integral of the form
\beq{eq:convolution}
(K\ast f_{\mr{in}})(x)= \int_{B\subset T_{x^*}M } K(v) f_{\mr{in}}\vert _{\exp_x v}(x) \,\rd^d v
\eeq
where $f_{\mr{in}}\vert _{\exp_x v}(x)$ 
is the parallel transport of $f_{\mr{in}}(\exp_x v)$
along the geodesic from $\exp_x v$ to  $x$, and $d$ the dimension of $M$.

The tangent space $T_{x^*}M$ is isomorphic to $\mb R^d$. A family of such isomorphisms parametrised by $M$ is called a gauge. A choice of gauge allows us to consider the kernel as a map
$$K:\mb R^d \to\Hom(V_{\mr{in}}, V_{\mr{out}})$$
where we keep notation light by using the same letter $K$, and renaming the fibres.

Because the choice of gauge is arbitrary, convolution has to be invariant under any gauge transformation in the structure group $G\subset \GL(d,\mb R)$. The action of the structure group on the feature spaces is denoted $\rho_{\mr{in}}:G\to\GL(V_{\mr{in}})$ and $\rho_{\mr{out}}:G\to\GL(V_{\mr{out}})$. Weiler et al.~\cite{weiler2021coordinate} show that the convolution is invariant under gauge transformations if an only if
\begin{equation}
    \label{eq:kernel constraint}
    K(g v)=\det(g)^{-1}\,\rho_{\mr{out}}(g) K(v)\rho_{\mr{in}}(g)^{-1},\quad g\in G,\quad v\in \mb R^d.
\end{equation}
Cheng et al.~\cite{cheng2019covariance} show that if the maps $\Gamma_{\mr{in}}\to\Gamma_{\mr{out}}$ between feature spaces satisfy some reasonable conditions, then they are of the form (\ref{eq:convolution}).

The Kernel constraint (\ref{eq:kernel constraint}) is solved for subgroups $G$ of $\NSO(2)$ by Weiler and Cesa \cite{weiler2019general} and for $G=\SO(3)$ by Weiler et al. \cite{weiler20183d}. For compact groups in general, the solution can be expressed in terms of harmonic basis functions, as Lang and Weiler explain \cite{lang2020wigner}. This leaves the constraint still to be solved for noncompact groups and for those compact groups for which we do not have harmonic basis functions readily available.

We propose to compute the space of $G$-equivariant kernels $K$ on $\mb R^d$ by considering the $G$-orbits in $\mb R^d$ separately. Although it is not always a trivial task to combine the results on these orbits in order to describe the kernels on $\mb R^d$ (see Appendix \ref{app:from homogeneous space to R^d} for worked out examples), the advantage is that the computation of equivariant kernels on the orbits, homogeneous spaces $\hs$, can be done with simple methods that work equally well for noncompact groups as for compact ones.

Therefore we will be looking for the $G$-equivariant maps $\hs\to V$ where $V=\det^{*}\otimes V_{\mr{out}}\otimes V_{\mr{in}}^*$. 
The special case $\det\otimes V_{\mr{in}}=V_{\mr{out}}$ comes with an associative algebra structure (matrix multiplication) and is therefore most interesting to us (the case $V_{\mr{in}}=V_{\mr{out}}$ also allows the use of biases in neural networks and this case is highlighted in the discussion of \cite{cheng2019covariance} due to interesting analogies in physics).

To ensure the integrals (\ref{eq:cross}) and (\ref{eq:convolution}) are finite, the kernels are often required to have compact support. We want to study equivariant kernels for noncompact groups $G$, which do not have compact support, so we do not impose this condition. The reader interested in convolution is advised to keep this in mind.

\vspace{3mm}
Automorphic Lie algebras were introduced in relation to integrable systems in mathematical physics by Lombardo and Mikhailov \cite{lombardo2004reductions, lombardo2005reduction}. They are Lie algebras of meromorphic maps from a Riemann surface $\hs$ into a finite dimensional Lie algebra $\mf{g}$ which are equivariant with respect to a group $\Gamma$ acting discretely on $\hs$ and on $\mf{g}$ by automorphisms, and have some restriction on the poles. The best known examples are the twisted loop algebras 
$$\mc{L}(\mf g, \sigma, m)=\left\{f:\mb C^\ast\to\mf g\,\,\text{Laurent polynomial}\,\vert\,f\left(e^{\frac{2\pi i}{m}} z\right)=\sigma f(z)\right\}$$
where $\sigma$ is an automorphism of order $m$ of the Lie algebra $\mf g$. Twisted loop algebras are used in concrete realisations of affine Kac-Moody algebras. These are however somewhat misleading examples because they are the only simple automorphic Lie algebras due to the landmark papers by Kac and Mathieu \cite{kac1968simple,mathieu1992classification}: loop algebras are highly exceptional.

The algebra of Onsager \cite{onsager1944crystal} gives a better impression of the nature of the topic. In the original work of Onsager, it is not at all obvious that this Lie algebra $\mf O$ fits in the framework of automorphic Lie algebras, but Roan showed that this is indeed the case \cite{roan1991onsager}, with the construction of an isomorphism
$$\mf{O} \cong\left\{f:\mb C^\ast\to \slnc[2] \,\,\text{Laurent polynomial}\,\vert\,f\left(z^{-1}\right)=\sigma f(z)\right\}$$
where $\slnc[2]$ is the Lie algebra of traceless complex $2\times 2$ matrices and $\sigma$ is an order $2$ automorphism thereof. The Onsager algebra is not simple, the commutator subalgebra $[\mf O, \mf O]$ is a nontrivial ideal. The involved action of the $2$ element group on $\mb C^\ast$  given by $z \mapsto z^{-1}$ has nontrivial stabilisers at $z=1$ and $z=-1$. The automorphic Lie algebra evaluates to a lower dimensional algebra at these points than elsewhere. This is the more typical situation. In fact, many automorphic Lie algebras with $\slnc[2]$ as target Lie algebra are isomorphic to the Onsager algebra, even if the involved group actions are more complicated. Examples abound in \cite{lombardo2010on, bury2021automorphic, 10.1093/imrn/rnab376, knibbeler2022classification}.

Before the recent work of Lombardo, Veselov and the author \cite{10.1093/imrn/rnab376}, the groups $\Gamma$ used in the construction of automorphic Lie algebras had been finite. The successful study of algebras of $\SLNZ[2]$-equivariant matrices shows that the infinite discrete groups are in reach, and there is an optimism that we can start to understand automorphic Lie algebras in various geometries.
A first step in this direction, studied in this paper, is to simplify the situation to the case of continuous (Lie) groups (which contain the desired discrete groups) and their homogeneous spaces. This enables us to obtain subalgebras of the desired automorphic Lie algebras.

There are methods to compute invariants of discrete groups from other invariants, such as transvection \cite{olver1999classical}, Rankin-Cohen brackets and modular derivatives \cite{bruinier2008the}. These methods have been proven effective in the development of automorphic Lie algebras by Lombardo, Sanders and the author \cite{knibbeler2017higher} and Lombardo, Veselov and the author \cite{10.1093/imrn/rnab376}, but they do have a weakness: you need an invariant to start with, before you can compute more of them. The subalgebras of automorphic Lie algebras that can be computed with the method of this paper provide a starting point for discrete methods and further study of automorphic Lie algebras in new geometries.

\vspace{3mm}
Various examples have been selected to illustrate the theory and inspire further development of geometric deep learning and automorphic Lie algebras.

\section{Invariant sections of homogeneous vector bundles}
\label{sec:vector}
Throughout this section, let $\hs$ be a homogeneous space of a Lie group $G$, and $V$ a representation of $G$. Our main practical contribution is this
\\[2mm]{\bf Elementary method to compute all $G$-equivariant maps $\hs\to V$.}
\vspace{-1mm}
\begin{description}
\item[1] Choose a base point $x_0\in \hs$.
\item[2]\label{it:H} Determine $H\coloneqq G_{x_0}=\{g\in G\,\vert \,g\,x_0=x_0\}$.
\item[3]\label{it:\hs} Find a map $f:\hs\to G$ such that $f(x)x_0=x$. 
\item[4]\label{it:V^H} Compute a basis $\{v_1,\ldots,v_n\}$ for the invariants $V^H$.
\end{description}
\vspace{-1mm}
The maps $x\mapsto f(x)v_1,\ldots,x\mapsto f(x)v_n$ constitute a basis for $\emap{G}{\hs}{V}$.

\vspace{3mm}
A few comments are in order. 
If $f:\hs\to G$ is such that $f(x)x_0=x$ then $x\mapsto f(x)H$ is an isomorphism $\hs\to G/H$. Existence of this map is equivalent to transitivity of the action of $G$ on $\hs$. It is well known that such a map can be chosen to be smooth near any point, but in general there is no such map that is globally smooth\footnote{There is a smooth map $f:\hs\to G$ such that $f(x)x_0=x$ if and only if the $H$-bundle $G\to \hs$ is isomorphic to the trivial bundle $\hs\times H\to \hs$. Indeed, given such $f$, one can construct a map $G\to \hs\times H$ by $g\mapsto(g x_0,f(g x_0)^{-1}g)$. This is a bijection with inverse $(x,h)\mapsto f(x)h$, and hence an isomorphism of $H$-bundles. For the converse, suppose that $\phi:\hs\times H\to G$ is an isomorphism of $H$ bundles, i.e.~a smooth bijection such that $\phi(x,h)x_0=x$. Then we can define a smooth map $f:\hs\to G$ is such that $f(x)x_0=x$ by $f(x)=\phi(x,1)$.}. 

It is not well known how to find a map like $f$ and we are in fact not aware of a general method to tackle this problem. Luckily it is easy to do for many of the cases that appear in geometric deep learning applications. It is for instance trivial for Euclidean spaces $\eucsp{n}\cong\SE(n)/\SO(n)$ and for spherical spaces $\sphsp{n}\cong\SO(n+1)/\SO(n)$ it can be solved with the Gram-Schmidt process (cf.~ Appendix \ref{app:spheres as homogeneous space}). For hyperbolic spaces we did not find a general approach. The $2$-dimensional case is simple enough and for the $3$-dimensional case we present an ad hoc solution in Example \ref{ex:hyperbolic space}.

The computation of the invariants $V^H$ is simplified by moving to the Lie algebra level. That is, one determines a basis $B$ for the Lie algebra $\mf{h}$ of the Lie group $H$ and solves the system of linear equations $h v=0$ for $v\in V$ and $h\in B$. The complicating factor of this step is that the dimension of $V$ can be very large in deep learning applications. Finzi, Welling and Wilson \cite{finzi2021practical} approached this problem using singular value decompositions, and they made their software library publicly available.

Arriving at the solution, we can also see why we do not impose conditions on the maps we consider. It is the action of $G$ that determines the properties of equivariant maps on homogeneous spaces. For example, if the action is smooth, so are the maps.

Before we prove that the method works, we give an example.
\begin{Example}[The hyperbolic plane $\hypsp{2}$]
\label{ex:H2}
The hyperbolic plane $\uh$ can be realised as the space of complex numbers $\tau=x+iy$ with positive imaginary part, $y>0$. The group $\SL(2,\mb{R})$ acts transitively on $\uh$ by M{\"o}bius transformations $$\begin{pmatrix}a&b\\c&d\end{pmatrix}\cdot \tau=\frac{a\tau+b}{c\tau+d}.$$
If we let $\SL(2,\mb{R})$ act on $\End(\mb{R}^2)$ by conjugation then we can define the space of $\SL(2,\mb{R})$-equivariant maps $\uh\to \End(\mb{R}^2)$. This example fits in the setting above with $\hs=\uh$, $\rg=\SL(2,\mb{R})$ and $V=\End(\mb{R}^2)$. To see this, we must check that $\SL(2,\mb{R})$ can take one point in $\uh$, say $i$, to an arbitrary point $\tau=x+iy\in\uh$. A group element that does this is
$$f(\tau)=\begin{pmatrix}\sqrt{y}&x/\sqrt{y}\\0&1/\sqrt{y}\end{pmatrix}.$$

To carry out our method we pick a point $i\in\uh$ to take on the role of $x_0$, and determine that the subgroup $H$ of $\SL(2,\mb{R})$ fixing $i$ is $\SO(2)$.
The map $f$ defined above satisfies $f(\tau)\cdot i=\tau$, so the only part that is possibly challenging is already taken care of. It remains to compute a basis for $\End(\mb{R}^2)^{\SO(2)}$. For instance we can take
$$v_1=\begin{pmatrix}1&0\\0&1\end{pmatrix},\quad v_2=\begin{pmatrix}0&-1\\1&0\end{pmatrix}.$$
We then arrive at a basis of the space of equivariant matrix-valued maps given by $\tau\mapsto f(\tau)v_i f(\tau)^{-1}$, $i=1,2$. This first basis element is given by the identity matrix and the second reads
\begin{align*}
\tau=x+iy\mapsto 
\begin{pmatrix}\frac{x}{y}&-\frac{x^2+y^2}{y}\\[2mm]\frac{1}{y}&-\frac{x}{y}\end{pmatrix}.
\end{align*}

As a bonus we have in this case a real algebra structure, where the product is the pointwise matrix multiplication. Identifying the above matrix with the imaginary unit provides an isomorphism to the algebra of complex numbers.
\end{Example}

We will now prove that this computational method gives the correct answers.
Maps $\hs\to V$ can be identified with sections of the trivial vector bundle $\hs\times V$. We will generalise this set-up to homogeneous vector bundles, following Cohen et al.~\cite{cohen2019a} and Aronsson \cite{aronsson2022homogeneous}, which are nontrivial bundles over $\hs$ defined by a representation of $H$ rather than $G$, known for their role in the celebrated Borel-Weil theorem. The invariant sections of homogeneous vector bundles are easily determined. These are then mapped onto the desired space of equivariant maps $\hs\to V$.

The following is a standard construction of a homogeneous vector bundle (see for instance \cite[Section 7.4.1]{sepanski2007compact}). 
Let $W$ be a $H$-module and define a right action of $H$ on $G\times W$ by $(g,w)h=(gh,h^{-1}w)$. The orbit space $(G\times W)/H$ is denoted
$$G\times_H W.$$ An element of this space is of the form $[g,w]=\{(gh,h^{-1}w)\,\vert \,h\in H\}$. Together with the projection map $$\pi :G\times_H W\to G/H,\quad [g,w]\mapsto gH$$ this is a $W$-bundle over $G/H$. It carries a $G$-action given by $g_1[g,w]=[g_1g,w]$. This action restricts to linear maps between fibres and descends to a transitive action on $G/H$. The vector space $\Gamma (G/H, G\times_H W)$ of sections $s$ becomes a $G$-module by the action $$(g_1 s)(gH)=g_1 s(g_1^{-1}gH).$$
Notice that we have started with a representation $W$ of $H$ and ended up with a representation of $G$. 

These representation are commonly reformulated in the language of maps. Let $\mapone{\hs}$ denote the maps $\hs\to\field$, where $\field$ is $\mb R$ or $\mb C$, 
and $\maptwo{\hs}{V}$ the maps $\hs\to V$. The sections of the homogeneous vector bundle $G\times_H W$ can be identified with the maps
$$
\mapthree{G}{H}{W}=
\{f:G\to W\,\vert \,f(gh)=h^{-1}f(g),\; \forall h\in H, \forall g\in G\}
$$ with $G$-action $(g_1\cdot f)(g)=f(g_{1}^{-1} g)$. This identification is realised by sending a map $f\in\mapthree{G}{H}{W}$ to the section $s_f:gH\mapsto [g,f(g)]$. To write down the inverse of this map, we must first relate the fiber at $eH$ of the homogeneous vector bundle to $W$. This is done with the $H$-intertwining bijection $\phi:\pi^{-1}(eH)\to W$ sending $[h,w]$ to $hw$. The inverse mapping from $\Gamma (G/H, G\times_H W)$ to $\mapthree{G}{H}{W}$ is then given by sending $s$ to $f_s:g\mapsto \phi(g^{-1}s(gH))$. This identification intertwines the $G$-action.

The sections of the trivial vector bundle $G/H\times V$ can be identified with the maps
$$
\mapthree{G}{H}{\field}\otimes V\cong\{f:G\to V\,\vert \,f(gh)=f(g),\; \forall h\in H, \forall g\in G\}$$
with $G$-action $(g_1\cdot f)(g)=g_1 f(g_{1}^{-1} g)$ (where the field $\field$ is considered as the trivial representation of $H$). This group action corresponds to the standard action on the tensor product $M\otimes V$ of two representation, given by $g(m\otimes v)=(gm)\otimes(gv)$, on the left hand side of the isomorphism. 

\begin{Remark}
If $G$ is compact then $\mapthree{G}{H}{W}$ is known as the representation of $G$ induced by the representation $W$ of $H$. For locally compact groups $G$ such as the Lie groups we work with, induced representations have been defined as various subrepresentations of $\mapthree{G}{H}{W}$. Mackey required functions to be $L^2$ (square integrable) with respect to a well chosen measure \cite{mackey1952induced}, and Moore worked with $L^1$ functions \cite{moore1962on}. For our simpler objective it is neither desirable nor necessary to make any of these restrictions.
\end{Remark}
The next result identifies the homogeneous vector bundles to the representation spaces where we look for automorphic Lie algebras and convolutional kernels for geometric deep learning.
\begin{Lemma}
\label{lem:homogeneous vb to trivial vb}
Let $V$ be a representation of $G$ and $W$ a representation of $H$. Let $\field$ denote the trivial representation of $H$. Then
$$\Phi:\mapthree{G}{H}{W}\to \mapthree{G}{H}{\field}\otimes V,\quad \Phi(f)(g)=gf(g)$$
is an isomorphism of $G$-representations if and only if $W$ is the restriction of $V$ to $H$.
\end{Lemma}
\begin{proof}Let us denote the restriction of $V$ to $H$ by $\Res{G}{H}{V}$.
Suppose that $W=\Res{G}{H}{V}$. Then $\Phi(f)\in \mapthree{G}{H}{\field}\otimes V$ whenever $f\in\mapthree{G}{H}{W}$ since $\Phi(f)(gh)=ghf(gh)=gf(g)=\Phi(f)(g)$, so that $\Phi$ is well defined. It has an inverse given by $\Phi^{-1}(\tilde{f})(g)=g^{-1}\tilde{f}(g)$ and the $G$-linearity follows from the computation 
\begin{align*}
\Phi(g_1\cdot f)(g)&=\Phi(f\circ g_1^{-1})(g)
\\&=gf(g_1^{-1} g)
\\&=g_1g_1^{-1}gf(g_1^{-1} g)
\\&=g_1\Phi(f)(g_1^{-1} g)
\\&=(g_1\cdot\Phi(f))(g),
\end{align*}
which completes the proof of one direction.

Suppose now that $\Phi$ is an isomorphism of $G$-modules. We show first that $W=V$ as vector spaces. Let $G=\sqcup_{x\in \hs}g_x H$ and for each $w\in W$ define the map $g_x h\mapsto h^{-1} w$, which is an element of $\mapthree{G}{H}{W}$. Therefore $\Phi$ sends it to an element $g_x h\mapsto g_x w \in V$. Taking $g_x=1$ we find that $W\subset V$.
In the other direction we have an easier construction because for each $v\in V$ the constant map $g\mapsto v$ is in $\mapthree{G}{H}{\field}\otimes V$. The inverse of $\Phi$ sends this map to the map $g\mapsto g^{-1}v$ in $\mapthree{G}{H}{W}$. With $g=1$ we see that $V\subset W$.

To see that the action of $H$ on $W$ is such that $W= \Res{G}{H}{V}$ we write out the equivariance condition for $g\mapsto f(g)=g^{-1}v$ in $\mapthree{G}{H}{W}$ while we keep track of the actions we are using with subscripts:
$$h\cdot_W (g^{-1}\cdot_Vv)=h\cdot_{W}f(g)=f(gh^{-1})=(gh^{-1})^{-1}\cdot_Vv=h\cdot_V (g^{-1}\cdot_Vv).$$
Putting $g=1$ we get $h\cdot_W v=h\cdot_V v$ and we conclude that $W=\Res{G}{H}{V}$.
\end{proof}

Lemma \ref{lem:aa1} to follow is a special instance of geometric Frobenius reciprocity.
It can be found in the excellent but challenging book by \u{C}ap and Slov\'{a}k, \cite[Theorem 1.4.4]{cap2009parabolic}\footnote{In \cite[Theorem 1.4.4]{cap2009parabolic}, the notation of sections $\Gamma(E)$ of a homogeneous bundle $E\to G/H$ is used in place of $\mapthree{G}{H}{W}$, and the fiber $E_o$ at the `base point' $o=H$ is used in place of $W$.}, where they use it to understand $G$-invariant geometric structures on $G/H$.
\begin{Lemma}
\label{lem:aa1}
The space $\mapthree{G}{H}{W}^G$ corresponds to the space of constant maps $G\to W^H$. This provides an isomorphism $\mapthree{G}{H}{W}^G\cong W^H$.
\end{Lemma}
\begin{proof}
If $f\in \mapthree{G}{H}{W}^G$ then $f(g)=g\cdot f(g)=f(1)$ hence $f$ is constant. If $f(g)=w$ then $hw=hf(g)=f(gh^{-1})=w$ hence $w\in W^H$ and $f$ is a constant map $G\to W^H$. 
Vice versa, any constant map $G\to W^H$ is clearly contained in $\mapthree{G}{H}{W}^G$.
\end{proof}

We will use the isomorphisms $\mapthree{G}{H}{\field}\cong \mapone{G/H}$ and $\mapthree{G}{H}{\field}\otimes V\cong \maptwo{G/H}{V}$ and the resulting isomorphism
$$(\mapthree{G}{H}{\field}\otimes V)^G\cong \emap{G}{\hs}{V},$$
which is nothing more than a change of perspective. To see this, notice that $\mapthree{G}{H}{\field}$ are the functions on $G$ which are constant on $H$-cosets. These correspond to the function on $G/H$ and are denoted $\mapone{G/H}$. Likewise, elements of $\mapthree{G}{H}{\field}\otimes V$ can be considered as maps $G\to V$ invariant on $H$-cosets, corresponding to maps $G/H\to V$ and denoted $\maptwo{G/H}{V}$. When we recall how $G$ acts on $\mapthree{G}{H}{\field}\otimes V$ we see that $G$-invariant elements of this space correspond to $G$-equivariant elements of $\maptwo{G/H}{V}$.

Using this reformulation, we can combine Lemmas \ref{lem:homogeneous vb to trivial vb} and \ref{lem:aa1} into the following theorem.
\begin{Theorem}
\label{thm:aa2}
Let $V$ be a representation of $G$.
The space of equivariant maps $\emap{G}{G/H}{V}$ corresponds to the algebra of maps $gH\mapsto gv$ where $v\in V^H$. This provides an isomorphism $\emap{G}{G/H}{V}\cong V^H$.
\end{Theorem}
\begin{proof}
We map the constant maps $G\to V^H$ that constitute $\mapthree{G}{H}{W}^G$ by Lemma \ref{lem:aa1} onto $\emap{G}{G/H}{V}$ using the isomorphism $\Phi$ of Lemma \ref{lem:homogeneous vb to trivial vb}.
\end{proof}
Theorem \ref{thm:aa2} reduces the infinite dimensional problem of determining the invariants in $\maptwo{G/H}{V}$ to the finite dimensional problem to determine $V^H$. It finishes the proof of our method to compute equivariant convolutional kernels.

We apply the method to a second common example.
\begin{Example}[The sphere $\sphsp{2}$]
The sphere is compact and the equivariant maps on the sphere are therefore described in the existing literature of geometric deep learning (cf.~\cite{weiler20183d}, \cite[Appendix E.4]{lang2020wigner} and \cite[Section 6]{gerken2023geometric}). This makes it a good example to illustrate our alternative approach to this problem.

For convenience we will work with the homogeneous space $$\overline{\mb{C}}\cong \PSU(2)/\PSU(1)$$ where $\PSU(2)=\SU(2)/\pm\Id$ acts on the Riemann sphere $\overline{\mb{C}}=\mb C \cup\{\infty\}$ by M{\"o}bius transformations, rather than the more obvious $\sphsp{2}\cong\SO(3)/\SO(2)$. Stereographic projection $\sphsp{2}\to\overline{\mb C}$ takes one description to the other.

The complex irreducible representations (irreps) of $\SU(2)$ can be realised as 
$$V_n=\mb C\langle x^n,x^{n-1}y, x^{n-2}y^2,\ldots,y^n\rangle$$
where an element of $g\in \SU(2)$ acts on a polynomial $p\in V_n$ by $(g\cdot p)(x,y)=p(g^{-1}(x,y))$. 
If $p\in V_{2m}$ then $p(-x,-y)=p(x,y)$ so that  $V_{2m}$ is a representation of $\SU(2)/\pm\Id\cong \SO(3)$. All irreps of $SO(3)$ are obtained this way.
Any complex representation of a compact group such as $\SU(2)$ and $\SO(3)$ is a direct sum of irreps.
See for instance \cite{sepanski2007compact} or another textbook on this topic.

We carry out our elementary method to compute $\emap{\PSU(2)}{\overline{\mb C}}{V_{2m}}$.
As base point we choose $x_0\coloneqq 0\in\overline{\mb C}$. The subgroup of $\PSU(2)$ stabilising $0$ is $$H=\left\{\pm\begin{pmatrix}a&0\\0&\overline{a}\end{pmatrix}\,\vert \,a\in\mb C, \vert a\vert =1\right\}.$$
We need a map $f:\overline{\mb C}\to\PSU(2)$ such that $f(Z)$ sends $0$ to $Z$. For instance, we can define\footnote{The map $f$ is not continuous at $Z=\infty$ but this poses no problem. A continuous solution does not exist because $\PSU(2)$ and $\overline{\mb C}\times\PSU(1)$ are topologically inequivalent.}
$$
f(Z)=
\left\{
\begin{array}{ll}
\pm\frac{1}{\sqrt{1+\vert Z\vert ^2}}\begin{pmatrix}1&Z\\-\overline{Z}&1\end{pmatrix}&\text{ if }Z\ne \infty,
\\[3mm]
\pm\begin{pmatrix}0&-1\\1&0\end{pmatrix}&\text{ if }Z=\infty.
\end{array}
\right.
$$
Our next step is to determine $V^H$. In this case we quickly find that 
$V_{2m}^H=\mb C x^m y^m$. Thus we find the solution: 
$\emap{\SU(2)}{\overline{\mb C}}{V_{2m}}$ is the one-dimensional space spanned by the map $Z\mapsto f(Z)x^my^m$ where
$$
f(Z)x^my^m=
\left\{
\begin{array}{ll}
\frac{(x-Zy)^m(\overline{Z}x+y)^m}{1+\vert Z\vert ^2}&\text{ if }Z\ne \infty,
\\[3mm]
(-1)^mx^my^m&\text{ if }Z=\infty.
\end{array}
\right.
$$

It was explained in \cite{weiler20183d,lang2020wigner} that this problem can be solved using knowledge of spherical harmonics (see e.g. \cite{toth2002finite}). With our approach we do not need to know bases of spherical harmonics a priori. Such bases appear automatically in the computation, because the coefficients of an equivariant vector form a basis for the spherical harmonics. To find these coefficients we rewrite
\begin{align*}
\frac{(x-Zy)^m(\overline{Z}x+y)^m}{1+\vert Z\vert ^2}
&=
\frac{1}{1+\vert Z\vert ^2}
\Big(\sum_{i=0}^m \binom{m}{i} x^{m-i}(-Zy)^{i}\Big)
\Big(\sum_{j=0}^m \binom{m}{j}(\overline{Z}x)^{m-j}y^j\Big)
\\&=
\frac{1}{1+\vert Z\vert ^2}
\sum_{i=0}^m \sum_{j=0}^m\binom{m}{i}\binom{m}{j} x^{2m-i-j} y^{i+j}
(-Z)^{i}\overline{Z}^{m-j}
\\&=
\sum_{\ell=0}^{2m}  x^{2m-\ell} y^{\ell}
\left[\frac{1}{1+\vert Z\vert ^2}
\sum_{i+j=\ell}\binom{m}{i}\binom{m}{j}
(-1)^i Z^{i}\overline{Z}^{m-j}
\right]
\end{align*}
and find in square brackets a basis for the spherical harmonics of degree $m$.
\end{Example}

We end this section with the important specialisation to tangent vector fields. 
The tangent bundle of a homogeneous space can be studied in the framework of homogeneous vector bundles, using the isomorphism  
\beq{eq:tangent bundle}
T(G/H)\cong G\times_H \mf g/\mf h
\eeq
of vector bundles over $G/H$ described by Sharpe in \cite[Chapter 4, Proposition 5.1]{sharpe1997differential}. The symbols $\mf g$ and $\mf h$ denote the Lie algebras of the groups $G$ and $H$ respectively. 
Combining Lemma \ref{lem:aa1} and (\ref{eq:tangent bundle}) we obtain
\begin{Corollary}
\label{cor:tangent vector fields}
The space of $G$-invariant tangent vector fields on $\hs$ is isomorphic to $(\mf g/\mf h)^H$. If $H$ is connected then one can write $(\mf g/\mf h)^H$ in purely Lie algebraic terms as $N_{\mf g}(\mf h)/\mf h$ where $N_{\mf g}(\mf h)$ is the normaliser $\{A\in\mf g\,\vert \,[A,B]\in\mf h\, \forall B\in\mf h\}$ of $\mf h$ in $\mf g$.
\end{Corollary}
When $H$ is the trivial group, Corollary \ref{cor:tangent vector fields} reduces to the well known fact that the space of left-invariant vector fields on a Lie group is isomorphic to its Lie algebra.

\section{Fixed point subalgebras of homogeneous algebra bundles}
\label{sec:algebra}
In this section we investigate what happens when the fibres of the homogeneous vector bundles have the structure of an algebra.

Throughout this section we assume that $G$ is a Lie group and $H$ a closed subgroup, and that $\ta$ is an $\field$-algebra and $H$-representation and these two structures are compatible in the sense that $h(aa_1)=(ha)(ha_1)$ for all $h\in H$ and $a,a_1\in\ta$.
\begin{Lemma}
$\mapthree{G}{H}{\ta}$ has an algebra structure defined by $(ff_1)(g)=f(g)f_1(g)$, where $f,f_1\in \mapthree{G}{H}{\ta}$ and $g\in G$. The invariants $\mapthree{G}{H}{\ta}^G$ form a subalgebra.
\end{Lemma}
We will take the algebra structure of $\mapthree{G}{H}{\ta}$ into account and write $\ma{G}{H}{\ta}$ to emphasize this. A subalgebra of invariants will be called an automorphic algebra. The automorphic algebra $\mapthree{G}{H}{\ta}^G$ will be denoted $\ema{G}{H}{\ta}$.

The linear isomorphisms obtained in the previous section are algebra isomorphisms in the setting of the current section. We state the results but omit the elementary proofs.
\begin{Theorem}
\label{thm:aa3} In the context of invariant sections of homogeneous bundles and equivariant maps we have  respectively the following results.
\begin{enumerate}[(a)]
\item
The automorphic algebra $\ema{G}{H}{\ta}$ corresponds to the algebra of constant maps $G\to \ta^H$. 
\item
If $G$ acts on $\ta$ by algebra isomorphisms, then the automorphic algebra $\aaa{G}{G/H}{\ta}$ corresponds to the algebra of maps $gH\mapsto ga$ where $a\in\ta^H$. 
\end{enumerate}
This provides isomorphisms $\ema{G}{H}{\ta} \cong \aaa{G}{G/H}{\ta}\cong \ta^H$ of algebras.

\end{Theorem}
\begin{Remark}
The isomorphism class of $\ema{G}{H}{\ta}$ does not depend on $G$. The full group only plays a role in the realisation of the algebra.
\end{Remark}

\begin{Example}[Automorphic Lie algebras on the Riemann sphere]
\label{ex:ls}
Lombardo and Sanders classified automorphic Lie algebras on the Riemann sphere with target algebra $\slnc[2]$ \cite{lombardo2010on} (elements of these Lie algebras are invariants of a finite group). Their analysis starts with the observation that the map $$\mb C^2\to \slnc[2],\quad \begin{pmatrix}x\\y\end{pmatrix}\mapsto\begin{pmatrix}xy&-x^2\\y^2&-xy\end{pmatrix}$$ is $\SLNC[2]$-equivariant with respect to left multiplication on $\mb{C}^2$ and the adjoint action (conjugation) on $\slnc[2]$.

We recover this observation with the approach of this paper, where we consider $\mb{C}^2\setminus \{0\}$ as  $\SLNC[2]$-space. The stabiliser of $x_0\coloneqq \begin{pmatrix}1\\0\end{pmatrix}$ is 
$$A=\left\{\begin{pmatrix}1&b\\0&1\end{pmatrix}\;\vert \;b\in\mb{C}\right\}=e^{\mb{C}\begin{pmatrix}0&1\\0&0\end{pmatrix}}.$$
A map $f:\mb{C}^2\setminus \{0\}\cong \SLNC[2]$ with $f(x)x_0=x$ is given by 
$$\begin{pmatrix}x\\y\end{pmatrix}\mapsto\begin{pmatrix}x&-y/r^2\\y&x/r^2\end{pmatrix},\quad r^2=x^2+y^2.$$ 
and provides an isomorphism $\mb{C}^2\setminus \{0\}\cong \SLNC[2]/A$.

The space of $A$-invariants in $\slnc[2]$ is spanned by $\begin{pmatrix}0&1\\0&0\end{pmatrix}$. Hence, by Theorem \ref{thm:aa3}, the space of  $\SLNC[2]$-equivariant maps $\mb{C}^2\setminus\{0\}\to\slnc[2]$ is spanned by 
$$\begin{pmatrix}x\\y\end{pmatrix}\mapsto\Ad\left(\begin{pmatrix}x&-y/r^2\\y&x/r^2\end{pmatrix}\right)\begin{pmatrix}0&1\\0&0\end{pmatrix}=\begin{pmatrix}-xy&x^2\\-y^2&xy\end{pmatrix}.$$

Theorem \ref{thm:aa3} is not necessary to compute this example. For example, if $V_n$ is the irreducible representation of $\SLNC[2]$ with highest weight $n$, then we know from representation theory that we have the isomorphisms of $\SLNC[2]$-modules
$$\End(V_n)\cong V_n\otimes V_n^\ast\cong V_n\otimes V_n\cong V_{2n}\oplus V_{2n-2}\oplus\ldots\oplus V_0.$$ With Schur's lemma we can then see that there is precisely one equivariant map $\mb{C}^2\setminus\{0\}\to\mf{gl}(V_n)$ of homogeneous order $0, 2, 4, \ldots, 2n$, up to multiplicative constant, and these can be computed by brute force. It is however difficult to say anything about the algebra structure of this space based only on the representation theory of $\SLNC[2]$. Here we find the value of Theorem \ref{thm:aa3} for this example, as it shows that the solution space is conjugate to the space of $A$-invariants.

If we define the nilpotent operator $E$ for a representation $\rho_n:\SLNC[2]\to\GL(V_n)$ by $$E=\rd\rho_n\begin{pmatrix}0&1\\0&0\end{pmatrix}$$ then any power of $E$ is fixed by $A$ (because $\rho_n(A)=e^{ \mb C E}$ and $\Ad(e^{ c E})E^n=e^{ c\, \ad(E)}E^n=E^n$). Moreover, the minimal polynomial of $E$ is $\lambda^{n+1}$, thus 
$$\mf{gl}(V_n)^A=\mb{C}[E]\cong \mb{C}[T]/(T^{n+1}),$$ and we have an isomorphism $$\mb{C}[T]/(T^{n+1})\to \aaa{\SLNC[2]}{\mb{C}^2\setminus\{0\}}{\mf{gl}(V_n)}$$ of algebras
given by
$$ T\mapsto \left[\begin{pmatrix}x\\y\end{pmatrix}\mapsto \Ad\left(\rho_n\begin{pmatrix}x&-y/r^2\\y&x/r^2\end{pmatrix}\right) E\right].$$
This cyclic algebra is of course a subalgebra of any automorphic Lie algebra defined by restricting the above actions of $\SLNC[2]$ to one of its subgroups.

If we replace $V_n$ by a reducible representation $U$ then the algebra $\mf{gl}(U)^A$ is no longer abelian.
\end{Example}

In the context of convolutional networks, the target representations are spaces of real matrices. 
In the context of automorphic Lie algebras we are interested in target spaces with algebra structures. For this reason we will emphasize the intersection of these special cases, where $W=\ta=\End(U)$. When we involve the algebra structure, we will write $\mf{gl}(U)$ instead of $\End(U)$.

Due to the Skolem-Noether theorem we know that any automorphism of $\mf{gl}(U)$ is inner, i.e. of the form $m\mapsto r m r^{-1}$ for some $r\in\GL(U)$. Hence any homomorphism $\rho:H\to\Aut{\mf{gl}(U)}$ is of the form $\rho(h)m=\sigma(h)m\sigma(h)^{-1}$ where $\sigma$ is a homomorphism $H\to\PGL(U)$, that is, a projective representation of $H$.

Recall that the endomorphism ring of an irrep is a division ring due to Schur's lemma, and that a real division ring is isomorphic to either the real numbers $\mb R$, the complex numbers $\mb{C}$ or the quaternions $\mb H$ (this statement is known as the Frobenius theorem). Real irreps are said to be of real type, complex type or quaternionic type accordingly. Combining these facts with the complete reducibility of representations of compact groups we obtain the following corollary of Lemma \ref{lem:aa1}.
\begin{Theorem}
\label{thm:endomorphisms}Let $K$ be a compact subgroup of $G$ and let $U$ be a real projective representation of $K$. Then there is an isomorphism of algebras $$\ema{G}{K}{\mf{gl}(U)}\cong \bigoplus_{r}\mf{gl}(r,\mb R)\oplus \bigoplus_{c}\mf{gl}(c,\mb C)\oplus \bigoplus_{h}\mf{gl}(h,\mb H)$$ where $r$ ranges over the multiplicities of projective $K$-irreps in $U_\mb{C}$ of real type, $c$ ranges over the multiplicities of conjugate pairs of projective $K$-irreps in $U_\mb{C}$ of complex type, and $h$ ranges over the multiplicities of pairs of projective $K$-irreps in $U_\mb{C}$ of quaternionic (or symplectic) type.
\end{Theorem}
We refer to Bourbaki \cite[Ch. VIII, \S 7, Proposition 12]{bourbaki2005lie} for a general method to determine the type of an irrep of a simple Lie algebra.

\begin{Example}[Common compact stabilisers] All two- and three-dimensional model geometries in the sense of Thurston \cite[\S 3.8]{thurston1997three} have compact stabilisers $K$ with connected component isomorphic to the trivial group or $\SO(2)$ or $\SO(3)$. The hyperbolic, spherical and Euclidean geometries $\hypsp{n}$, $\sphsp{n}$ and $\eucsp{n}$ have stabilisers $\NSO(n)$ in their isometry groups and stabilisers $\SO(n)$ in their orientation preserving isometry groups.
These cases are therefore expected to be most common in applications and deserve special attention.

The group $\SO(2)$ has only one irrep of real type: the trivial representation. All nontrivial irreducible projective representations of $\SO(2)$ are of complex type. Hence $$\ema{G}{\SO(2)}{\mf{gl}(U)}\cong \mf{gl}(r,\mb R)\oplus \bigoplus_{c}\mf{gl}(c,\mb C)$$ where $r=\dim U^{\SO(2)}$, cf.~Example \ref{ex:H2}.

The irreducible projective representations of $\SO(3)$ correspond to the irreps of its double cover $\SU(2)$. There is precisely one irrep of $\SU(2)$ of any finite dimension. It is of real type when the dimension is odd (and the representation descends to a representation of $\SO(3)$) and of quaternionic type if the dimension is even, so that
$$\ema{G}{\SO(3)}{\mf{gl}(U)} \cong  \bigoplus_{r}\mf{gl}(r,\mb R)\oplus \bigoplus_{h}\mf{gl}(h,\mb H).$$
\end{Example}

\begin{Example}[Quaternions]
\label{ex:quaternions}
Let $V_1$ be the natural representation of $\SU(2)$ of $2$ complex-dimensions. It is of quaternionic type. Hence $V_1\oplus V_1$ has a basis in which $\SU(2)$ is represented by matrices with real entries. Let $U$ be the corresponding ($4$ real-dimensional) representation. Then $\mf{gl}(U)$ is a real representation of $\SU(2)/\pm\Id\cong \SO(3)$ (and of its extensions). We write down a basis for $\mf{gl}(U)^{\SO(3)}$ to confirm it is isomorphic to the quaternions. 

With the embedding $\mb{C}^\ast\mapsto\GL(2,\mb{C})$ given by $a+bi\mapsto\begin{pmatrix}a&b\\-b&a\end{pmatrix}$ we construct real matrices
$$\rho_U\left(\begin{pmatrix}a+bi&-c+di\\c+di&a-bi\end{pmatrix}\right)=
\begin{pmatrix}a&b&-c&d\\-b&a&-d&-c\\c&d&a&-b\\-d&c&b&a\end{pmatrix}$$
for the representation $\rho_U:\SU(2)\to\GL(U)$.
The matrices that commute with the image of this representation are spanned by the identity matrix and 
$$I=\begin{pmatrix}0&-1&0&0\\1&0&0&0\\0&0&0&-1\\0&0&1&0\end{pmatrix},\quad
J=\begin{pmatrix}0&0&-1&0\\0&0&0&1\\1&0&0&0\\0&-1&0&0\end{pmatrix},\quad
K=\begin{pmatrix}0&0&0&-1\\0&0&-1&0\\0&1&0&0\\1&0&0&0\end{pmatrix}.$$
These matrices satisfy $I^2=J^2=K^2=IJK=-\Id$ and provide the isomorphism $\mf{gl}(U)^{\SO(3)}\cong \mb H$.
\end{Example}

We end this section with a detailed description of an interesting automorphic algebra illustrating Theorem \ref{thm:aa3}. This example has a second goal of illustrating the fact that we can hardwire Lorentz group-equivariance into a convolutional network even if the group is not compact. This can be compared to the approach of Bogatskiy et al.~\cite{bogatskiy2020lorentz} where a Lorentz group-equivariant network is used for the study of particle physics. The equivariance is incorporated as another learning task in this network, rather than being hardwired.

\begin{Example}[Hyperbolic space]
\label{ex:hyperbolic space}
In this example we will compute the algebra of maps from hyperbolic $3$-space to the algebra of matrices $\mf{gl}(U)$ from Example \ref{ex:quaternions}, which are equivariant under the Lorentz group, and exhibit its isomorphism to the quaternion algebra.

Hyperbolic $3$-space $\hypsp{3}$ will be identified with the hyperboloid model;
$$\hypsp{3}=\{(t,x,y,z)\in\mb R_{\ge 1}\times \mb R^3\,\vert \,-t^2+x^2+y^2+z^2=-1\},$$
a component of a level set of the quadratic form 
$$Q(t,x,y,z)=-t^2+x^2+y^2+z^2.$$
The group of linear transformations preserving $Q$ is denoted $\NSO(1,3)$ and corresponds to the group of automorphisms of $\hypsp{3}$. The connected component $G\coloneqq \SO^+(1,3)$ containing the identity  acts transitively on $\hypsp{3}$ (as we will find below) and the stabiliser subgroups are isomorphic to $\SO(3)$.

There is an isomorphism of real Lie groups between $\SO^+(1,3)$ and $\PSL(2,\mb C)$, which can be understood using the Weyl representation that we now describe.
The space of $2\times 2$ hermitian matrices of determinant $1$ has two connected components distinguished by the sign of the trace. Its component of matrices with positive trace is in one-to-one correspondence with Hyperbolic space by the map
$$\phi:(t,x,y,z)\mapsto A=\begin{pmatrix} t+z&x-iy\\x+iy&t-z\end{pmatrix}.$$
The group $\PSL(2,\mb C)$ acts on this space by $$[g]A=gAg^*$$ where $g^*$ is the conjugate transpose of $g$ (it is easy to see that $gAg^*$ is hermitian and has determinant $1$. To argue that its trace is positive we can use the fact that $\SLNC[2]$ is connected, so that $gAg^*$ is in the same connected component as $\Id\,A\,\Id^*=A$). When this action of $\PSL(2,\mb C)$ is translated to $\hypsp{3}$ using $\phi$ one finds an isomorphism $\PSL(2,\mb C)\to \SO^+(1,3)$. We will continue this example using the Weyl representation.

We choose a base point $x_0\coloneqq (1,0,0,0)\in\hypsp{3}$. Its stabiliser is $H\coloneqq \PSU(2)$. To find a map $f:\hypsp{3}\to \PSL(2,\mb C)$ such that $f(x)x_0=x$, we need to solve 
\beq{eq:Lorentz}
gg^*=A
\eeq for $g\in \SL(2,\mb C)$. 
Since $A$ is hermitian, we can write it as $$A=PDP^*$$ with $D=\diag(\lambda^+,\lambda^-)$ where $$\lambda^{\pm}=t\pm\sqrt{t^2-1}$$ are the eigenvalues of $A$ and $P\in \SU(2)$ (uniquely defined up to right multiplication by $\diag(\zeta,\zeta^*)$, $\zeta\in\NSU(1)$).
We then find a solution to (\ref{eq:Lorentz}) given by $$g=P\sqrt{D}P^*$$ where $\sqrt{D}$ is defined as $\diag(\sqrt{\lambda^+},\sqrt{\lambda^-})$, and hence an explicit isomorphism $$\hypsp{3}\to \PSL(2,\mb C)/\PSU(2),\quad (t,x,y,z)\mapsto \pm g \PSU(2).$$ The matrix $g$ thus constructed reads
\beq{eq:g}
g=\frac{1}{\sqrt{2}\sqrt{t+1}}\begin{pmatrix}
t+1+z
&
x-iy
\\
x+iy
&
t+1-z
\end{pmatrix}.
\eeq

Let us write down a basis for the automorphic algebra  $\aaa{\SO(1,3)^+}{\hypsp{3}}{\mf{gl}(U)}$ where $U$ is the representation defined in Example \ref{ex:quaternions} in which we computed a basis $\{1,I,J,K\}$ for $\mf{gl}(U)^{\SO(3)}$. 

The group element (\ref{eq:g}) acts on $U$ with the matrix
$$\rho_U(g)=\frac{1}{\sqrt{2}\sqrt{t+1}}\begin{pmatrix}
t+1+z&0&x&-y\\
0&t+1+z&y&x\\
x&y&t+1-z&0\\
-y&x&0&t+1-z
\end{pmatrix}$$
and on $\mf{gl}(U)$ by conjugation with this matrix. Applying $g$ to $I,J$ and $K$ respectively yields
\begin{align*}
\mc I&=\left(\begin{array}{rrrr}
0 & -1 & 0 & 0 \\
1 & 0 & 0 & 0 \\
0 & 0 & 0 & -1 \\
0 & 0 & 1 & 0
\end{array}\right)
\\[2mm]
\mc J
&=\left(\begin{array}{rrrr}
x & y & -t-z & 0 \\
y & -x & 0 &t+z \\
t-z & 0 & -x & y \\
0 & -t+z & y & x
\end{array}\right)
\\[2mm]
\mc K&=\left(\begin{array}{rrrr}
-y & x & 0 & -t-z \\
x & y & -t-z & 0 \\
0 & t-z & -y & -x \\
t-z & 0 & -x & y
\end{array}\right)
\end{align*}
Together with the identity matrix, these form a basis for $\aaa{\SO(1,3)^+}{\hypsp{3}}{\mf{gl}(U)}$ due to Theorem \ref{thm:aa3}. The quaternion algebra can thus be presented by matrix-valued functions on hyperbolic space equivariant under the Lorentz group.
\end{Example}

\section{Discussion}
\label{sec:discussion}
We have described how the problem of computing all vector-valued equivariant maps on a homogeneous space can be reduced to finite dimensional linear algebra with a geometric version of Frobenius reciprocity.

If $G$ is compact then one can take the alternative approach of Fourier analysis to find the invariants. Representations of compact groups are completely reducible, so that $L^2(\hs)$ and $V$ decompose into irreps, and each pair of isomorphic irreps in these respective modules gives rise to one invariant in $L^2(\hs)\otimes V\cong L^2(\hs,V)$. Such an approach is presented by Lang and Weiler \cite{lang2020wigner} and Lombardo, Sanders and the author in \cite{knibbeler2017higher}.
The hard part of this approach is the determination of the decomposition of $L^2(\hs)$. For important cases the solution is readily available. E.g.~the decomposition of $L^2(\SO(2))$ is known from classical Fourier analysis and $L^2(\SO(3)/\SO(2))$ from harmonic analysis on the two dimensional sphere.

The approach presented in this paper has the advantage that $G$ can be noncompact, Fourier analysis is not needed, irreps of $G$ need not be determined and neither do isotypical components of $L^2(\hs)$. 
This simple approach is possible because the domains $\hs$ considered here are $G$-orbits, hence any equivariant map is determined by one value, contrary to the situation of automorphic Lie algebras where $\hs/G$ is a Riemann surface.
This also shows the disadvantage of our approach. If $\hs$ is not a homogeneous space, there is more work to be done (see Appendix \ref{app:from homogeneous space to R^d}).

For the study of algebra structures our construction has a second advantage: it preserves the algebra structure. For this reason there is no need to compute structure constants and find normal forms after the invariants are determined, which is one of the major challenges in the research on automorphic Lie algebras to date.

Another natural approach to solve a Lie group-invariance constraint is to start with differentiation and reduce it to a linear problem defined by the generators of the Lie algebra. Finzi, Welling and Wilson \cite{finzi2021practical} used this approach in the context of geometric deep learning. The authors show how the resulting linear system can be effectively solved using the singular value decomposition in case the representation is finite dimensional, and they made their software library publicly available. For the infinite dimensional representation $\mapthree{G}{H}{W}$ considered in this paper, and commonly in geometric deep learning, the linearised problem is a system of partial differential equations. Solving this system is much harder than the method presented in this paper, and is therefore no longer practical. However, our results reduce the problem of finding $G$-invariants in the infinite dimensional function space to finding $H$-invariants in the finite dimensional representation $W$ of $H$. For high-dimensional representations $W$ the approach and software of \cite{finzi2021practical} can be used to compute $W^H$ and complete the computation of $\mapthree{G}{H}{W}^G$.

We hope that the simplicity of the presented method, and its applicability to noncompact groups, will make it useful to the designers of convolutional neural networks with symmetries. 
Furthermore, we hope that this material shall serve as a starting point for the study of automorphic Lie algebras in unexplored geometries such as hyperbolic $3$-space or Siegel upper-half space, and uncover new Lie structures.

\appendix
\section{The Gram-Schmidt process and $\sphsp{n}\cong\SO(n+1)/\SO(n)$}
\label{app:spheres as homogeneous space}
In order to compute the space of $G$-equivariant maps from a homogeneous space $X\cong G/H$ to a representation $V$ of $G$, we propose to find an explicit isomorphism $X\cong G/H$ first. That is, obtain a map $f:X\to G$ such that $f(x)x_0 = x$ for any $x\in X$, where $x_0$ is a `base point' of our choice. As we mentioned in Section \ref*{sec:vector}, we are not aware of a general method to construct such a map (in the mathematical literature, one is usually satisfied knowing that it exists, and there is no need for a construction), but for specific families of homogeneous spaces there are constructions available. In this section we solve the question for the spheres
\[\sphsp{n}_r=\{(x_1,\ldots,x_{n+1})\in\mb{R}^{n+1}\,|\,x_1^2+\ldots+x_{n+1}^2=r^2\}\cong\SO(n+1)/\SO(n),\quad r>0,\] 
using the Gram-Schmidt process. 

To warm up, we consider $\sphsp{1}_r=\{(x,y)\in\mb{R}^{2}\,|\,x^2+y^2=r^2\}$. When we pick the base point $x_0=(r,0)$ we need to find $f_1:\sphsp{1}_r\to \SO(2)$ (the subscript of $f$ refers to the dimension of the sphere) such that $f_1(x,y)(r,0)^T = (x,y)^T$. That is, 
\[f_1(x,y)=\begin{pmatrix}x/r&\ast\\y/r&\ast\end{pmatrix}.\] 
It remains to replace the $\ast$ by entries which ensure that $f_1(x,y)\in\SO(2)$. In this case, we quickly find the solution 
\[f_1(x,y)=\begin{pmatrix}x/r&-y/r\\y/r&x/r\end{pmatrix}.\] 

Going one dimension higher, we consider $\sphsp{2}_r=\{(x,y,z)\in\mb{R}^{3}\,|\,x^2+y^2+z^2=r^2\}$ and pick the base point $x_0=(r,0,0)$. Then we need to construct an orthogonal matrix $f_2(x,y,z)$ with first column $(x/r,y/r,z/r)^T$. This time it is harder to find a solution. We will return to this problem shortly.

The Gram-Schmidt process takes a set of linearly independent vectors $\{v_1,\ldots,v_k\}$ in a vector space with inner product $\langle\cdot,\cdot\rangle$ and returns a set of orthogonal vectors $\{u_1,\ldots,u_k\}$ with the same linear span. If we write the projection of $v$ on a nonzero vector $u$ as 
\[P_u(v)=\frac{\langle u,v\rangle}{\langle u,u\rangle}u,\] then the orthogonal set of vectors is given by
\begin{align*}
    u_1&=v_1\\
    u_2&=v_2-P_{u_1}(v_2)\\
    u_3&=v_3-P_{u_1}(v_3)-P_{u_2}(v_3)\\
    &\vdots
    \\u_k&=v_k-\sum_{i=1}^{k-1} P_{u_i}(v_k).
\end{align*}
To transform this orthogonal set $\{u_1,\ldots,u_k\}$ into an orthonormal set $\{e_1,\ldots,e_k\}$ one computes $e_i=u_i/\sqrt{\langle u_i, u_i\rangle}$. 

Let us return to the problem of constructing an orthogonal matrix $f_2(x,y,z)$ with first column $(x/r,y/r,z/r)^T$. A solution to this problem is not unique. Indeed, any solution can be multiplied on the right by a matrix of the form 
\[
    \begin{pmatrix}
        1&0&0\\0&\cos \theta&-\sin\theta\\0&\sin\theta&\cos \theta
    \end{pmatrix}
\]
because this multiplication preserves the first column, orthogonality and the determinant. Using this wiggle room we can ensure that $f_2$ is of the form
\[
    f_2(x,y,z)=
    \begin{pmatrix}
        x/r&\ast&0\\
        y/r&\ast&\ast\\
        z/r&\ast&\ast    
    \end{pmatrix}.
\]
Next we construct orthogonal columns with the Gram-Schmidt process. Due to the zero on the top right of the matrix $f_2(x,y,x)$, we can save ourselves some work by computing columns from right to left. At the right, we start with a vector $(0,1,0)^T$. Instead of computing its component orthogonal to $(x/r,y/r,z/r)$, we use  $(0,y/r,z/r)$ in order to preserve the first zero. That is, we compute $(0,1,0)-P_{(0,y/r,z/r)}(0,1,0)=\left(0,\frac{z^2}{y^2+z^2},\frac{-yz}{y^2+z^2}\right)$ (which is only defined when $y^2+z^2>0$). This vector normalises to $\left(0,\frac{z}{\sqrt{y^2+z^2}},\frac{-y}{\sqrt{y^2+z^2}}\right)$ and is used for the right column of $f_2$.

For the middle column, we start with a vector $(1,0,0)^T$ which is already orthogonal to the right column, leaving us only to compute its component orthonormal to the left column: $(1,0,0)-P_{(x/r,y/r,z/r)}(1,0,0)=\left(\frac{y^2+z^2}{r^2},-\frac{xy}{r^2},-\frac{xz}{r^2}\right)$, which normalises to $\left(\frac{\sqrt{y^2+z^2}}{r},\frac{-xy}{r\sqrt{y^2+z^2}},\frac{-xz}{r\sqrt{y^2+z^2}}\right)$, resulting in the matrix
\newcommand{\myvspace}{2mm}
\[
    f_2(x,y,z)=
    \begin{pmatrix}
        \frac{x}{r}&\frac{\sqrt{y^2+z^2}}{r}&0\\[\myvspace]
        \frac{y}{r}&\frac{-xy}{r\sqrt{y^2+z^2}}&\frac{z}{\sqrt{y^2+z^2}}\\[\myvspace]
        \frac{z}{r}&\frac{-xz}{r\sqrt{y^2+z^2}}&\frac{-y}{\sqrt{y^2+z^2}}    
    \end{pmatrix},\quad y^2+z^2>0.
\]
By construction, we have $f_2(x,y,z)\in\NSO(3)$, and we can check that the determinant of $f$ is $1$ hence $f_2(x,y,z)\in\SO(3)$. It remains to define $f_2$ for the case $y^2+z^2=0$, i.e. $x=\pm r$. Here we can simply put $f_2(r,0,0)=\Id$ and $f_2(-r,0,0)=\diag(-1,-1,1)$.

In much the same way, we can construct a map $f_n:\sphsp{n}_r\to \SO(n+1)$ solving $f(x)x_0=x$, where $x=(x_1,x_2,\ldots,x_{n+1})$. Generally, the equation leaves the freedom to multiply a solution $f(x)$ by an element of the stabiliser subgroup $G_{x_0}=H$ on the right. In our particular case we choose $x_0=(r,0,0,\ldots,0)^T$ and the stabiliser subgroup consists of block-diagonal matrices of the form $\diag(1,A)$ where $A\in\SO(n)$. It can be used to ensure that the first row of $f_n(x_1,\ldots,x_{n+1})$ has the form $(x_1/r, \ast, 0,0,\ldots,0)$. Likewise, an element $\diag(1,1,A)\in H$ where $A\in\SO(n-1)$ can be used to ensure the second row of $f_n(x_1,\ldots,x_{n+1})$ has the form $(x_2/r, \ast, \ast, 0,0,\ldots,0)$, without changing the first row. By continuing this argument for the whole sequence of subgroups $\SO(n)\supset \SO(n-1)\supset\ldots\supset \SO(2)$ of the stabiliser, we can assume without loss of generality that $f_n$ takes the form 
\[
  f_n(x_1,\ldots,x_{n+1})=
  \begin{pmatrix}
    x_1/r&\ast&0&0&0&\ldots&\ldots&0\\
    x_2/r&\ast&\ast&0&0&\ldots&\ldots&0\\
    x_3/r&\ast&\ast&\ast&0&\ldots&\ldots&0\\
    \vdots&\vdots&\vdots&\vdots&\ddots&\ddots&&\vdots\\
    \vdots&\vdots&\vdots&\vdots&&\ddots&\ddots&\vdots\\
    x_{n-1}/r&\ast&\ast&\ast&\ldots&\ldots&\ast&0\\
    x_{n}/r&\ast&\ast&\ast&\ldots&\ldots&\ast&\ast\\
    x_{n+1}/r&\ast&\ast&\ast&\ldots&\ldots&\ast&\ast\\
  \end{pmatrix}.  
\]
We will compute orthonormal columns for this matrix from right to left. For the right most column, we start with the vector $(0,\ldots,0,1,0)$ and obtain a norm 1 vector orthogonal to $(0,\ldots,0,x_{n}/r,x_{n+1}/r)$ by computing $(0,\ldots,0,1,0)-P_{(0,\ldots,0,x_{n}/r,x_{n}/r)}(0,\ldots,0,1,0)$ 
and subsequently normalising to get $\left(0,\ldots,0,\frac{x_{n+1}}{\sqrt{x_{n}^2+x_{n+1}^2}},-\frac{x_{n}}{\sqrt{x_{n}^2+x_{n+1}^2}}\right)$.

Consider now the column of $f_n(x)$ at position $k+1$, and assume all columns to its right fit the form above, and constitute an orthogonal set of vectors together with the vector $(0,\ldots,0,x_k/r,\ldots,x_{n+1}/r)$.
We can apply the Gram-Schmidt process for this set of vectors together with the vector $(0,\ldots,0,1,0,\ldots,0)^T$ where the nonzero entry is at position $k$, and see that all but one of the projections to be computed are zero, leaving us only to compute
$(0,\ldots,0,1,0,\ldots,0)-P_{(0,\ldots,0,x_k/r,\ldots,x_{n+1}/r)}((0,\ldots,0,1,0,\ldots,0))$ which equals 
\[\frac{1}{\sum_{i=k}^{n+1}x_i^2}(0,\ldots,0,\sum_{i=k+1}^{n+1}x_i^2,-x_kx_{k+1},-x_kx_{k+2},\ldots,-x_kx_{n+1}).\] 
It is convenient to introduce the notation
\[r_k=\sqrt{\sum_{i=k}^{n+1}x_i^2},\quad k=1,\ldots,n,\]
because the norm of the last vector is $r_{k+1}/r_k$, and its normalisation reads
\[\left(0,\ldots,0,\frac{r_{k+1}}{r_k},-\frac{x_k x_{k+1}}{r_k r_{k+1}},\ldots,-\frac{x_k x_{n+1}}{r_k r_{k+1}}\right).\]
Notice that $r_{n}=\sqrt{x_n^2+x_{n+1}^2}>0$ implies $r_{i}>0$ for $i=1,\ldots,n$, hence all these vectors are defined if and only if $x_n^2+x_{n+1}^2> 0$. 

Using these vectors, for $k=1,\ldots,n$, as columns for $f_n(x)$, we obtain for $x_n^2+x_{n+1}^2> 0$
\[
  f_n(x)=
  \begin{pmatrix}
    \frac{x_1}{r_1}&\frac{r_2}{r_1}&0&0&0&\ldots&\ldots&0\\[\myvspace]

    \frac{x_2}{r_1}&-\frac{x_1 x_{2}}{r_1 r_{2}}&\frac{r_3}{r_2}&0&0&\ldots&\ldots&0\\[\myvspace]

    \frac{x_3}{r_1}&-\frac{x_1 x_{3}}{r_1 r_{2}}&-\frac{x_2 x_{3}}{r_2 r_{3}}&\frac{r_4}{r_3}&0&\ldots&\ldots&0\\[\myvspace]

    \vdots&\vdots&\vdots&\vdots&\ddots&\ddots&&\vdots\\[\myvspace]

    \vdots&\vdots&\vdots&\vdots&&\ddots&\ddots&\vdots\\[\myvspace]

    \frac{x_{n-1}}{r_1}&-\frac{x_1 x_{n-1}}{r_1 r_{2}}&-\frac{x_2 x_{n-1}}{r_2 r_{3}}&-\frac{x_3 x_{n-1}}{r_3 r_{4}}&\ldots&\ldots&\frac{r_{n}}{r_{n-1}}&0\\[\myvspace]

    \frac{x_{n}}{r_1}&-\frac{x_1 x_{n}}{r_1 r_{2}}&-\frac{x_2 x_{n}}{r_2 r_{3}}&-\frac{x_3 x_{n}}{r_3 r_{4}}&\ldots&\ldots&-\frac{x_{n-1} x_{n}}{r_{n-1} r_{n}}&(-1)^{n+1}\frac{x_{n+1}}{r_n}\\[\myvspace]

    \frac{x_{n+1}}{r_1}&-\frac{x_1 x_{n+1}}{r_1 r_{2}}&-\frac{x_2 x_{n+1}}{r_2 r_{3}}&-\frac{x_3 x_{n+1}}{r_3 r_{4}}&\ldots&\ldots&-\frac{x_{n-1} x_{n+1}}{r_{n-1} r_{n}}&(-1)^{n}\frac{x_n}{r_n }\\
  \end{pmatrix}.  
\]
In the right column of this matrix we added a factor $(-1)^{n+1}$ in order to have determinant $1$. To see this, notice that the function $\det\circ f_n$ is continuous and takes values in the discrete set $\{1,-1\}$ due to orthonormality, hence it is constant on the connected domain $\{x_n^2+x_{n+1}^2> 0\}=\sphsp{n}\setminus\sphsp{n-2}$, and we can compute the value of $\det\circ f_n$ by evaluating in one point. For instance, $\det f_n(0,\ldots,0, r)=1$. Thus, we have $f_n(x)\in\SO(n+1)$ when $x_n^2+x_{n+1}^2> 0$.

We conclude using induction. If $x_n^2+x_{n+1}^2=0$ then $(x_1,\ldots,x_{n-1})\in\sphsp{n-2}$ and we define $f_n(x)$ as the block diagonal matrix $\diag(f_{n-2}(x),1,1)$. Given that we already defined $f_1$ and $f_2$, this provides $f_n$ for all $n\ge 1$.

It is worth noting that $f_n$ is constant on lines emanating from the origin. That is, $f_n(\lambda x)=f_n(x)$ for any $\lambda\ge 0$ and any $x\ne 0$. All values of $f_n$ are therefore achieved on $\sphsp{n}_1$. 

\section{From homogeneous space to $\mb{R}^d$}
\label{app:from homogeneous space to R^d}
This appendix concerns the application of this paper to geometric deep learning and is intended as a starting point for future research.

In geometric deep learning, one encounters a group $G\in \GL(\mb{R}^d)$ and is interested in $G$ equivariant maps from $\mb{R}^d$ to a representation $V$ of $G$, as we explained in the introduction. In the body of this paper we restrict the problem to $G$-orbits in $\mb{R}^d$. Whether this can be extended to the whole space depends on the particular group and particularly the quotient space $\mb R^d/G$. 

We solve the problem for the spheres in $\mb{R}^d$ with Proposition \ref{prop:SO(n) equivariant maps on R^d}, where the quotient space is a manifold (appearing as $[0,\infty)$ in the proposition). A similar result for the case $d=3$ can be found in \cite{weiler20183d}. Afterwards we discuss the situation for hyperbolic spaces in $\mb{R}^d$, where the quotient space is not Hausdorff. 

This illustrates that future research should not aim for a general solution but rather focus on the application. In geometric deep learning, the equivariant maps are used as convolution kernels in an integral. Hence, it is harmless to exclude a set of measure zero, and one could restrict to a full measure subset of the quotient space $\mb R^d/G$ which is better behaved, allowing also solutions for hyperbolic spaces in $\mb{R}^d$.

In the setting of this appendix, we can choose what class of functions we want to study, unlike the situation for homogeneous spaces, where the function class is determined by the group actions. This appendix is written for continuous maps, but the proofs can be modified to work for other classes of maps.

\begin{Proposition}
    \label{prop:SO(n) equivariant maps on R^d}
    Let $V$ be a continuous representation of $\SO(d)$, $H\subset \SO(d)$ the stabiliser of $(1,0,\ldots,0)\in\mb{R}^d$, and $\{v_1,\ldots,v_m\}$ a basis of $V^H$, extending a basis $\{v_1,\ldots,v_{m'}\}$ of $V^{\SO(d)}$.
    Then the space of continuous $\SO(d)$-equivariant maps from $\mb{R}^d$ to $V$ is given by 
    \[
        C_{\SO(d)}\left(\mb{R}^d,V\right)=
        \left\{x\mapsto
            \sum_{i=1}^m c_i(r) f_{n}(x) v_i\,|\,c_i\in C^{}([0,\infty),\mb{R})
        ,\,c_i(0)=0\text{ if }i>m'
        \right\}
    \]
    where $x=(x_1,\ldots,x_d)\in\mb{R}^d$ and $r=\sqrt{x_1^2+\ldots+{x_d}^2}$, and $f_{n}:\mb{R}^d\to\SO(d)$ the map defined in Appendix \ref{app:spheres as homogeneous space} with $n=d-1$, extended to the origin by $f_{n}(0)=\Id$.
\end{Proposition}
\begin{proof}
    The set on the right-hand side of the equation will be denoted RHS, that is
    \[
        \text{RHS}=
        \left\{x\mapsto
        \sum_{i=1}^m c_i(r) f_{n}(x) v_i\,|\,c_i\in C^{}([0,\infty),\mb{R})
        ,\,c_i(0)=0\text{ if }i>m'
        \right\}.
    \]
    We will show first that RHS is contained in $C_{\SO(d)}\left(\mb{R}^d\setminus\{0\},V\right)$. To check that a map $x\mapsto\sum_{i=1}^m c_i(r) f_{n}(x) v_i$ in RHS is equivariant, we pick a generic element $g\in\SO(n)$ and compute
    \begin{align*}
        \sum_{i=1}^m c_i(gr) f_{n}(gx) v_i
        &=
        \sum_{i=1}^m c_i(r) f_{n}(gx) v_i\\
        &=
        \sum_{i=1}^m c_i(r) gf_{n}(x) v_i\\
        &=g\sum_{i=1}^m c_i(r) f_{n}(x) v_i
    \end{align*} 
    where the second equality follows from Theorem \ref{thm:aa2}.
    
    The next thing to do is to check continuity (and we will find that the discontinuity of $f_{n}$ is irrelevant). First pick a nonzero $x\in\mb{R}^d$ and restrict to the sphere $\sphsp{n}$ containing $x$ (and centered at the origin). There, $f_{n}(x) v_i$ is continuous due to equivariance and the properties of the group actions. Indeed, the limit $\lim_{y\to x}f_{n}(y) v_i$ with $y\in\sphsp{n}$ can be written as $\lim_{g\to 1}f_{n}(gx) v_i$ with $g\in\SO(d)$ because the action of $\SO(d)$ on $\sphsp{n}$ is continuous and transitive. Equivariance equates this to $\lim_{g\to 1}\left(gf_{n}(x) v_i\right)$ and continuity of the action on $V$ shows this limit is $\left(\lim_{g\to 1}g\right)f_{n}(x) v_i=f_{n}(x) v_i$.

    To see that $c_i(r)f_{n}(x) v_i$ depends continuously on $x$ at any nonzero $x$, recall that $f_{n}(\lambda x)=f_{n}(x)$ for any $\lambda>0$, and $c_i$ is continuous.

    For continuity at $x=0$, we use the fact that $v_i\in V^{\SO(d)}$ when $i\le m'$ which implies $c_i(r)f_{n}(x) v_i=c_i(r)v_i$ where continuity follows from continuity of $c_i$. Let now $i>m'$. Then $\lim_{x\to 0}c_i(r)f_{n}(x)v_i=0$ since $c_i$ is continuous, $c_i(0)=0$ and all entries of $f_{n}(x)\in\SO(d)$ are absolutely bounded by $1$. This limit corresponds to the value $c_i(0)f_{n}(0)v_i$. Thus, we see that the sum $\sum_{i=1}^m c_i(r) f_{n}(x) v_i$ is indeed continuous on $\mb R^d$, and $\text{RHS}\subset C_{\SO(d)}\left(\mb{R}^d,V\right)$.

    It remains to prove that the reverse inclusion $C_{\SO(d)}\left(\mb{R}^d,V\right)\subset \text{RHS}$ holds. Let $F$ be an element of $C_{\SO(d)}\left(\mb{R}^d\setminus\{0\},V\right)$. Then $F(r,0,\ldots,0)\in V^H$ for any $r\ge 0$, since $hF(r,0,\ldots,0)=F(h(r,0\ldots,0))=F(rh(1,0\ldots,0))=F(r,0\ldots,0)$. Hence, we can expand $F(r,0,\ldots,0)$ in a basis of $V^H$, i.e. there exists functions $c_i$ such that 
        \[
        F(r,0,\ldots,0)=\sum_{i=1}^m c_i(r)v_i.
    \]
    The functions $c_i$ must be continuous because $F(r,0,\ldots,0)$ depends continuously on $r$. Moreover, $c_i(0)=0$ when $i>m'$ since $F(0)\in V^G$ due to equivariance.
    Define a map by $D(x)=F(x)-\sum_{i=1}^m c_i(\sqrt{x_1^2+\ldots+{x_d}^2}) f_{n}(x)v_i$. This map is equivariant by the first half of this proof, and it vanishes at $(r,0,\ldots,0)$ for any $r\ge 0$, since $f_{n}(r,0\ldots,0)=\Id$. Let $x\in\mb{R}^d\setminus\{0\}$ and $r=\sqrt{x_1^2+\ldots+{x_d}^2}$. Then there exists $g\in\SO(d)$ such that $g(r,0,\ldots,0)=x$. Therefore, $D(x)=D(g(r,0\ldots,0))=gD(r,0,\ldots,0)=g 0=0$. That is, $F(x)=\sum_{i=1}^m c_i(r) f_{n}(x)v_i$ and $F\in \text{RHS}$.
\end{proof}

Hyperbolic space of dimension $n$ can be embedded in $\mb R^d$, $d=n+1$, as one component of the level set at $-1$ of the quadratic form
\[Q(t,x_1,\ldots,x_n)=-t^2+x_1^2+\ldots+x_n^2.\]
This is an orbit of the connected group $\SO^+(1,n)\subset \GL(d,\mb R)$ defined by preservation of $Q$. 

We treated different realisations of $2$- and $3$-dimensional hyperbolic space in Example \ref{ex:H2} and Example \ref{ex:hyperbolic space} respectively. Now we start with the $1$-dimensional example. In that case we have the quadratic form $Q(t,x)=-t^2+x^2$ which is preserved by the group
\[
    \SO^+(1,1)=\left\{\begin{pmatrix}\cosh \theta&\sinh \theta\\\sinh \theta&\cosh \theta\end{pmatrix}\,|\,\theta\in\mb R\right\}.
\]
Each nonzero point $(t,0)$ and $(0,x)$ is contained in a unique orbit of this group. This describes all but five of the orbits. The remaining orbits constitute the level set $\{Q=0\}=\{t^2=x^2\}$: four diagonal rays and the origin. The latter five orbits cannot be separated by open sets, hence $\mb R^2/\SO^+(1,1)$ is not Hausdorff and therefore not a manifold, and we cannot do calculus on this space. However, if we take away the diagonals $t^2=x^2$, the remaining quotient space $(\mb R^2\setminus\{t^2=x^2\})/\SO^+(1,1)$ is homeomorphic to the union of four open intervals, which is a manifold. Given that the use of the equivariant maps in deep learning for this example, is in an integral over $\mb R^2$, it is harmless to omit the set $\{t^2=x^2\}$ of measure zero in $\mb R^2$ from the domain of these equivariant maps.

To describe continuous $\SO^+(1,1)$-equivaiant maps $\mb R^2\setminus\{Q\ne 0\}\to V$, for any continuous representation $V$ of $\SO^+(1,1)$, we can proceed as we did in Proposition \ref{prop:SO(n) equivariant maps on R^d}. The equivariant maps are given by the sums 
\[\sum_{i=1}^m c_i(t,x)f(t,x)v_i\]
where each $c_i$ is a continuous function, and each $c_i(t,x)$ depends only on the $\SO^+(1,1)$-orbit containing $(t,x)$, and where
$f(t,x)$ is the matrix $\frac{1}{\sqrt{-Q}}\begin{pmatrix}t&x\\x&t\end{pmatrix}$ when $Q<0$ and 
$\frac{1}{\sqrt{Q}}\begin{pmatrix}x&t\\t&x\end{pmatrix}$ when $Q>0$.

In higher dimensions, the quotient space $\mb R^d/\SO^+(1,n)$ behaves similarly. The level set $\{Q=0\}=\{t^2=\langle x,x\rangle\}$ consists of $3$ orbits of $\SO^+(1,n)$, with $t=0$, $t>0$ and $t<0$. These orbits cannot be separated by open sets, and $\mb R^d/\SO^+(1,n)$ is not Hausdorff.
At positive values of $Q$, the level set is one orbit of $\SO^+(1,n)$, and at negative values of $Q$, the level set consists of two orbits (hyperbolic spaces). The quotient of the full measure subset $(\mb R^d\setminus\{Q=0\})$ by the action of $\SO^+(1,n)$ is homeomorphic to the union of three open intervals.

\section*{Acknowledgements}
We are grateful to Jan E.~Gerken, Sara Lombardo, Fahimeh Mokhtari, Jan A.~Sanders, Alexander Veselov and Maurice Weiler for helpful and stimulating discussions, 
and an anonymous reviewer for their thorough and thought-provoking feedback.

\section*{Declarations}
\subsection*{Funding and/or Competing interests.}
This work is partially supported
by the London Mathematical Society through an Emmy Noether Fellowship, REF EN-2122-03. The author has no relevant financial or non-financial
interests to disclose.

\printbibliography
\end{document}